\documentclass{amsart}
\usepackage[margin=1in]{geometry}
\usepackage{amsmath,amssymb,amsthm}
\usepackage{mathrsfs}
\usepackage{graphicx}
\usepackage{soul}
\usepackage{xcolor}
\usepackage[hypertexnames=false]{hyperref}
\usepackage[semibold]{sourceserifpro}
\usepackage{subcaption}

\usepackage{tikz-cd}
\usepackage{tikz}
\usetikzlibrary{calc}
\usetikzlibrary{arrows.meta}
\tikzset{>=Stealth}

\usepackage[backend=biber,
        style=alphabetic,citestyle=alphabetic,
        maxnames=99,
        url=false,isbn=false,doi=false,
        hyperref=true,backref=true]{biblatex}

\DeclareSourcemap{
  \maps[datatype=bibtex]{
    \map{
      \step[fieldset=primaryClass, null]
    }
  }
}
\DefineBibliographyStrings{english}{
backrefpage = {$\uparrow$},
backrefpages = {$\uparrow$}
}
\usepackage[capitalize]{cleveref}
\usepackage{thmtools}
\theoremstyle{plain}
\numberwithin{equation}{section}
\newtheorem{theorem}{Theorem}[section]
\newtheorem{corollary}[theorem]{Corollary}
\newtheorem{conjecture}[theorem]{Conjecture}
\newtheorem{lemma}[theorem]{Lemma}
\newtheorem{proposition}[theorem]{Proposition}
\theoremstyle{definition}
\newtheorem{definition}[theorem]{Definition}
\newtheorem{remark}[theorem]{Remark}
\newtheorem{example}[theorem]{Example}
\newtheorem*{theorem*}{Theorem}
\newenvironment{manualtheorem}[1]{%
\renewcommand\thetheorem{#1}
\begin{theorem}
}{\end{theorem}\addtocounter{theorem}{-1}}

\newcommand{\R}{\mathbb{R}}
\newcommand{\A}{\mathbb{A}}
\renewcommand{\P}{\mathbb{P}}
\newcommand{\Z}{\mathbb{Z}}

\newcommand{\C}{\mathbb{C}}
\newcommand{\Q}{\mathbb{Q}}

\newcommand{\ip}[1]{\left\langle #1 \right\rangle}
\newcommand{\gen}[1]{\left\langle #1 \right\rangle}

\newcommand{\word}[1]{\textbf{#1}}
\DeclareMathOperator{\GL}{GL}
\DeclareMathOperator{\Hom}{Hom}
\DeclareMathOperator{\Spec}{Spec}
\DeclareMathOperator{\Proj}{Proj}

\DeclareMathOperator{\conv}{conv}

\DeclareMathOperator{\supp}{supp}

\DeclareMathOperator{\im}{img}

\usepackage{todonotes}

\title{Toric varieties modulo reflections}
\author{Colin Crowley, Tao Gong, and Connor Simpson}
\begin{document}

\maketitle

\begin{abstract}
  Let $W$ be a finite group generated by reflections of a lattice $M$. 
  If an $M$-rational polytope $P \subset M \otimes_{\Z}\R$ is preserved by 
  $W$, then we show that the quotient of the projective toric variety 
  $X_P$ by $W$ is isomorphic to the toric variety $X_{P \cap D}$, where 
  $D$ is a fundamental domain for the action of $W$. This answers a 
  question of Horiguchi-Masuda-Shareshian-Song, and recovers 
  results of Blume, of the second author, and of Gui-Hu-Liu. 
  We also study quotients of real toric varieties, proving that $X_P^\R / W$ is contractible when $P$ is a permutohedron.
\end{abstract}

\section{Introduction}
\subsection{} The quotient of a toric variety by a finite group need not 
be toric, even when the group acts by toric automorphisms.
For instance, the quotient of $\P^n \times \P^n$ by the $\Z/2$ action 
that swaps the factors is not toric when $n \geq 2$ \cite[Remark 
8.2]{HMJJ21}, and the quotient of a 47-torus over $\Q$ by the $\Z/47$ action that 
cycles its factors is not even rational \cite[Theorem 
1]{S69}\footnote{We thank Laurent Moret-Bailly for writing the 
\href{https://mathoverflow.net/questions/297470/is-the-quotient-of-a-toric-variety-by-a-finite-group-still-toric}{MathOverflow 
post} that made us aware of this example.}!
The existence of these examples is unsurprising: it is almost never possible to make the quotient map torus-equivariant, so there is little reason to expect a toric structure on the quotient.

Nonetheless, Lorenz \cite{L05} provides a class of finite group actions on tori with toric quotient.
Let $V$ be a finite-dimensional real vector space with a positive definite bilinear form $\gen{,}$.
A \word{reflection} in $V$ is a linear transformation of the form
\[
  s_\alpha(v) := v - 2 \frac{\ip{v, \alpha}}{\ip{\alpha,\alpha}} \alpha
\]
for some nonzero $\alpha \in V$. A \word{reflection group} in $V$ is a finite subgroup of $\GL(V)$ generated by reflections.
A finite reflection group $W$ in $V$ has a fundamental domain that is a polyhedral cone $D \subset V$.

\begin{theorem}\cite[Theorem 6.1.1]{L05}\label{lorenz}
  If $W \subset \GL(V)$ is a finite reflection group preserving a lattice $M \subset V$, then there is an isomorphism $\Z[D \cap M] \to \Z[M]^W$. 
\end{theorem}
The semigroup $D \cap M$ is an example of a \word{saturated affine semigroup}, i.e. a finitely generated $\Z_{\geq 0}$-submodule $S \subset M$ such that for all $k \in \Z_{\geq 0}$ and $m \in M$, $km \in S$ implies $m \in S$.
Saturated affine semigroup algebras are precisely the coordinate rings of normal affine toric varieties \cite[Theorem 1.3.5]{CLS11}, so we may restate \cref{lorenz} as:
\begin{manualtheorem}{\ref{lorenz}$'$}\cite[Theorem 6.1.1]{L05}\label{lorenz:torus}
    If $W \subset \GL(V)$ is a reflection group preserving a lattice $M \subset V$, and $T = \Spec \Z[M]$, then there is an isomorphism of algebraic varieties $T / W \cong \Spec \Z[D \cap M]$.
\end{manualtheorem}
We emphasize that the quotient map is essentially never toric.

\subsection{} Our first result (illustrated in \cref{quotient-example}) is a generalization of
Lorenz's theorem. 
 
\begin{theorem}\label{main}
If $W \subset \GL(V)$ is a reflection group preserving a lattice $M \subset V$ and a saturated affine semigroup $S \subset M$, then
  there is an isomorphism $\Z[D \cap S] \to \Z[S]^W$.
\end{theorem}
Equivalently, we may state \cref{main} as:
\begin{manualtheorem}{\ref{main}$'$}\label{main:affine}
If $W \subset \GL(V)$ is a reflection group preserving a lattice $M \subset V$ and a saturated affine semigroup $S \subset M$, and $X = \Spec \Z[S]$, then
$X/W$ is isomorphic to the normal affine toric variety $\Spec \Z[D \cap S]$.
\end{manualtheorem}

Projective toric varieties with dense torus $\Spec \Z[M]$ are given by full-dimensional $M$-rational polytopes $P \subset V$\cite[Chapter 2]{CLS11}.\footnote{An $M$-rational polytope is a polytope with all vertices in a fractional multiple of $M$. We use rational polytopes instead of lattice polytopes, because the latter class is not preserved under intersection with $D$.}
A consequence of \cref{main} is the following statement, conjectured in \cite[Question 8.3]{HMJJ21}.

\begin{corollary}\label{main:projective}
Let $W \subset \GL(V)$ be a reflection group preserving a lattice $M \subset V$ and a full-dimensional $M$-rational polytope $P \subset V$, with associated projective toric variety $X_P$. Then the quotient $X_P / W$ is isomorphic to the projective toric variety $X_{P \cap D}$.
\end{corollary}

The quotients we refer to in \cref{lorenz:torus,main:affine,main:projective} are ``geometric quotients'' in the sense of Geometric Invariant Theory (GIT), which always exist for the action of a finite group on a quasi-projective variety \cite[Example 6.1]{D03}. Over the complex numbers, the analytic topology of $X/W$ is the quotient topology of the $W$-action on the analytic topology of $X$ by \cite[Remark 1.7]{neemanTopologyQuotientVarieties1985}.

\medskip
Horiguchi, Masuda, Shareshian, and Song were motivated to conjecture \cref{main:projective} after establishing an isomorphism of rational cohomology rings $H^*(X_P/W; \Q) \cong H^*(X_{P \cap D}; \Q)$ in the case when $W$ is a parabolic subgroup of an irreducible reflection group $W'$ of classical type, and $P$ is a \word{$W'$-permutohedron} (i.e.\ the vertices of $P$ are $\{w(u) : w \in W'\}$, with $u \in M$ a point on which $W'$ acts freely) \cite[Theorem 1.1]{HMJJ21}.
Gui-Hu-Liu \cite[Theorem 1.2]{GHL24} and Song \cite[Theorem 1.1]{song2022toric} subsequently established similar cohomological isomorphisms when $P$ is a permutohedron of any Lie type, and when $P$ is two-dimensional, respectively.
Meanwhile, Gong \cite[Theorems 1.1, 1.4]{G24} obtained a topological result, proving that $X_P/W$ and $X_{P \cap D}$ are homotopy equivalent when $P$ is a $W'$-permutohedron and $W \subset W'$ is a parabolic subgroup.
All of these statements are recovered by \cref{main:projective}.

Further special cases of \cref{main:projective} include the well-known isomorphisms
$(\P^1)^n / S_n \cong \P^n$ and
$\P^{n-1} / S_n \cong \P(1,2, \ldots, n)$,
as well as results of Blume \cite[Theorems 3.1, 5.3, 5.9]{B15}, which state that the projective toric variety associated to a permutohedron of type A, B, or C modulo the corresponding Weyl group is toric.


\subsection{}
We next study the topological quotient $X_P^\R/W$ of the real points $X_P^\R$ of $X_P$.
The algebraic results described above are of little avail in this setting, as the GIT quotient of a variety by a finite group may not coincide with the topological quotient when working over a field that is not algebraically closed.

Emulating \cite{song2022toric}, we first study the surface case.
\begin{theorem}\label{quotients of real surfaces}
  If $\dim P = 2$ and $W$ is the Weyl group of an irreducible crystallographic root system of rank 2, then $X_P^\R / W$ is homotopy equivalent to $\bigvee_{\ell-1} S^1$, where $\ell$ is the number of vertices of $P$ that lie in $D$.
\end{theorem}
We then consider the setting of \cite{GHL24}, in which $P$ is a $W$-permutohedron.
The topology of real permutohedral varieties is well-studied and combinatorially intricate (see, e.g. \cite{henderson2012rational,cho2019geometric,choiCohomologyRingsReal2024}).
In contrast, we show:

\begin{theorem}\label{main-thm-real}
  If $W$ is the Weyl group of an irreducible crystallographic root system, $M$ is its root lattice, and $P$ is a $W$-permutohedron, then $X_{P}^{\R}/W$ is contractible.
\end{theorem}

An immediate consequence of \cref{main-thm-real} is:
\begin{corollary}\label{trivial-rep-real}
  In the notation of \cref{main-thm-real}, the trivial representation of 
  $W$ is not a subrepresentation of $H^i(X_P^{\R}; \C)$ for $i > 0$.
\end{corollary}
\cref{trivial-rep-real} follows in types A and B from works of \cite{henderson2012rational} and \cite{cho2019geometric}, respectively, on the $\Q[W]$-module structure of $H^*(X_P^\R;\Q)$.
(We remark that these module structures also imply that quotients of $X_P^\R$ by parabolic subgroups of $W$ may have non-trivial topology. Describing these quotients remains an interesting problem.)

  Motivated by Theorems \ref{quotients of real surfaces} and \ref{main-thm-real}, we make the following conjecture on the topology of $X_P^\R/W$.
  \begin{conjecture}\label{wedgeconj}
    Suppose $W$ is the Weyl group of an irreducible crystallographic root system that spans $V$, and $M$ is its root lattice.
    Then $X_P^\R / W$ is homotopy equivalent to a wedge of spheres, and has Euler characteristic $\chi(X_P^\R / W) = 2-\ell$, where $\ell$ is the number of vertices of $P$ that lie in $D$.
  \end{conjecture}

  Our proofs of Theorems \ref{main}, \ref{main:affine}, and \cref{main:projective}, uniform with respect to $W$, $M$, and $P$, stand starkly against the ad-hoc approach we employ to prove Theorems \ref{quotients of real surfaces} and \ref{main-thm-real}. We hope that \cref{wedgeconj} will inspire the development of a more conceptual approach to the latter results.

\subsection*{Organization}
In \cref{groups}, we review finite reflection groups, then establish \cref{main,main:affine,main:projective} in \cref{proofs}.
In the remainder of the paper, we turn our attention to real points. We review topological models for $X_P^\R$ and $X_P^\R / W$, then prove \cref{quotients of real surfaces} in \cref{sec:realpts}. Finally, in \cref{Higher dimensional cases}, we study quotients of real permutohedral varieties and prove \cref{main-thm-real}.

\subsection*{Acknowledgements}
We are grateful to Gui Tao for making us aware of \cite[Question 8.1]{HMJJ21}, and to Hiraku Abe \& Haozhi Zeng for comments that improved the exposition.
We also thank Chiara Damiolini \& Sam Payne for discussing early 
versions of this project. TG is indebted to Matthias Franz for much help, and to Suyoung Choi, Ajneet Dhillon and Nicole Lemire  for discussions.
CC and CS thank the Simons Foundation for support; CC is also supported by NSF grant DMS-2039316.

\section{Finite reflection groups}\label{groups}
We review the theory of finite reflection groups, following \cite{H90}.
\subsection{} Let $V$ be a real vector space equipped with a positive definite symmetric bilinear form $\ip{,}$.
For $R = \Z_{\geq 0}, \Z, \R$, we write $R\gen{v_1, \ldots, v_k}$ for the $R$-span of vectors $v_1, \ldots, v_k$ (the use of angle brackets for both span and inner product should cause no confusion).
The \word{reflection} by a nonzero $\alpha \in V$ is the linear transformation given by 
\[
  s_{\alpha}(v) := v - 2 \frac{\ip{v, \alpha}}{\ip{\alpha,\alpha}} \alpha.
\]
A \word{reflection group} in $V$ is a subgroup of $\GL(V)$ generated by reflections.
All reflection groups appearing in this paper are assumed finite.

\subsection{}
A \word{root system} is a finite set $\Phi \subset V$ such that for all $\alpha \in \Phi$,
\begin{enumerate}
\item $\phi \cap \R\gen{\alpha} = \{ \alpha, -\alpha\}$, and
\item $s_\alpha(\Phi) = \Phi$.
\end{enumerate}
Elements of $\Phi$ are \word{roots}.
If $\Phi$ is a root system, then the subgroup $W$ of $\GL(V)$ generated by $\{s_\alpha : \alpha \in \Phi\}$ is a finite reflection group, and all such groups arise in this manner.
Any root system contains a \word{simple system}, that is,  is a linearly independent subset $\Delta \subset \Phi$ such that
\begin{enumerate}
\item $\Delta$ and $\Phi$ span the same subspace of $V$, and
\item every element of $\Phi$ is real linear combination of elements of $\Delta$ with coefficients all of the same sign.
\end{enumerate}
Elements of $\Delta$ are \word{simple roots}. Properties of $\Delta$ include:
\begin{enumerate}
\item $W$ is generated by $\{s_\alpha : \alpha \in \Delta\}$  \cite[Theorem 1.5]{H90}.
\item \label{fundamental}  $D = \{ v \in V : \ip{v,\alpha} \geq 0, \; \alpha \in \Delta\}$ is a fundamental domain for the action of $W$ on $V$ (i.e. each $W$-orbit intersects $D$ in exactly one point) \cite[Theorem 1.12]{H90}.
\end{enumerate}

\begin{remark}
  The definition of ``reflection'' used here differs from the one used in \cite{L05}, but the two definitions coincide under the hypotheses of \cref{lorenz}.
  Lorenz states \cref{lorenz} for a subgroup $W$ of $\GL(M)$ generated by transformations $s$ such that the fixed space of $s$ on $\Q \otimes M$ has codimension 1.
  Since $s$ lies in $\GL(M)$, it has determinant $\pm 1$; therefore, $s$ is conjugate over $\Q$ to a diagonal matrix with entries $1, 1, 1, \ldots, 1, -1$.
  Replacing $\ip{,}$ with $\sum_{(w_1, w_2) \in W \times W} \ip{w_1 -, w_2-}$, we may assume that the inner product on $V = \R \otimes M$ is $W$-invariant.
  If $\alpha$ is a $-1$-eigenvector of $s$ and $v$ is fixed by $s$, then
  \[ 2\ip{\alpha,v} = \ip{\alpha,v} + \ip{s\alpha, sv} = \ip{\alpha,v} + \ip{-\alpha, v} = 0, \]
  so $s = s_\alpha$, which is a reflection as defined here.
\end{remark}
\begin{remark}
  The definition of ``root system'' used here is that of \cite{H90}, but differs slightly from the definition used in some other sources.
  To assure the reader that this discrepancy causes no problem, we will provide proofs for results from references other than \cite{H90}.
\end{remark}
\subsection{}\label{groups:ordering}
Fix a simple system $\Delta$ for $W$, with corresponding fundamental domain $D$.
Define a partial order on $V$ by $u \geq v$ if and only if $u-v$ is a non-negative real linear combination of simple roots.
\begin{proposition}\cite[Lemma 1.12]{H90}\label{groups:ordering:orbit}
  If $v \in D$, then $v - wv \in \R_{\geq 0}\gen{\alpha_1, \ldots, \alpha_r}$ for all $w \in W$. In particular, $v \geq wv$.
\end{proposition}
Let $\conv(A)$ denote the convex hull of a subset $A \subset V$.
\begin{proposition}\cite[Proposition VI.2.4(ii)]{BtD85}\label{groups:ordering:convex}
  If $u, v \in D$, then $u \geq v$ if and only if $\conv(Wu)$ contains $v$.
\end{proposition}
\begin{proof}
  If $v \in \conv(Wu)$, then there are non-negative real numbers $c_w$, $w \in W$, such that $\sum_w c_w = 1$ and $v = \sum_w c_w wu$.
  By \cref{groups:ordering:orbit}, $wu \leq u$, so
  \[
    v = \sum_w c_w wu \leq \sum_w c_w u = u.
  \]
  Conversely, suppose $v \not \in \conv(Wu)$. Then there is $l \in V$ such that for all $w \in W$, $\ip{v, l} > \ip{wu, l} = \ip{u, w^{-1} l}$.
  In particular, if $l' = w^{-1} l \in D$, then by \cref{groups:ordering:orbit},
  \[
    \ip{v, l'} \geq \ip{w^{-1} v, l'} = \ip{v, wl'} > \ip{u, l'},
  \]
  hence $\ip{u - v, l'} < 0$.
  This proves that $u \not \geq v$ because an element of $\R_{\geq 0}\gen{\alpha_1, \ldots, \alpha_r}$ has non-negative inner product with all elements of $D$.
\end{proof}

\subsection{}\label{groups:crystallographic} Call a finite reflection group $W$ \word{crystallographic} if it preserves a lattice in $V$ (i.e. the $\Z$-span of a basis of $V$).
Given a simple system $\Delta = \{\alpha_1, \ldots, \alpha_r\}$ for $W$, the associated \word{weight group} is
\[
  \Lambda = \{ v \in V : 2 \frac{\ip{v, \alpha_i}}{\ip{\alpha_i, \alpha_i}} \in \Z, \, 1 \leq i \leq r \}.
\]
The weight group decomposes as
\[
  \Lambda = Z \oplus \Z\gen{\lambda_1, \ldots, \lambda_r},
\]
where $Z = \cap_{i=1}^r \alpha_i^\perp$ and the vectors $\lambda_1, \ldots, \lambda_r \in Z^\perp$ are the \word{fundamental weights}, determined by
\[
  2 \frac{\ip{\lambda_i, \alpha_j}}{\ip{\alpha_j, \alpha_j}} = \begin{cases} 1, & i=j\\ 0, & \text{else.}\end{cases}
\]
The fundamental weights are linearly independent.
Likewise, the fundamental domain decomposes as
\[
  D = Z \times \R_{\geq 0}\gen{\lambda_1, \ldots, \lambda_r}.
\]

\begin{proposition}\cite[\S15]{speyernotes} \label{goodroots}
  Let $W$ be a finite reflection group that preserves a lattice $M \subset V$.
  There is a simple system $\Delta$ for $W$ such that
  \[
    \Delta \subset M \subset \Lambda.
  \]
  Moreover, the action of $W$ preserves $\Lambda$.
\end{proposition}
\begin{proof}
  Let $\Delta_0$ be a simple system for $W$.
  If $\alpha \in \Delta_0$ and $m \in M \setminus \alpha^\perp$, then
  \[
    M \ni m - s_\alpha(m) = 2 \frac{\ip{m,\alpha}}{\ip{\alpha,\alpha}} \alpha,
  \]
  so $M \cap \R \gen{\alpha} \neq \{0\}$. Consequently, $M \cap \R \gen{\alpha}$ is a rank 1 lattice with a unique generator $\alpha'$ satisfying $\ip{\alpha,\alpha'} > 0$.
  The set $\Delta = \{\alpha' : \alpha \in \Delta_0\}$ is a simple system for $W$ that is contained in $M$.

  The weight lattice $\Lambda$ associated to $\Delta$ contains $M$ because for any $\alpha \in \Delta$ and $m \in M$,
  \[
    \Z\gen{\alpha} = \R \gen{\alpha} \cap M \ni m - s_\alpha(m) = 2 \frac{\ip{m,\alpha}}{\ip{\alpha,\alpha}} \alpha.
  \]
  The action of $W$ preserves $\Lambda$ because for each $\alpha \in \Delta$ and $\lambda \in \Lambda$,
  \[
    s_\alpha(\lambda) = \lambda - 2 \frac{\ip{\lambda,\alpha}}{\ip{\alpha,\alpha}}\alpha \in \Lambda + \Z\gen{\alpha} = \Lambda,
  \]
  and the reflections associated to simple roots generate $W$.
\end{proof}

\section{Proofs of \cref{main,main:affine}, and \cref{main:projective}}\label{proofs}
Throughout this section, we fix notations:
\begin{itemize}
\item $V$ is a real vector space with positive definite symmetric bilinear form $\ip{,}$,
\item $M \subset V$ is a lattice,
\item $W \subset \GL(V)$ is a finite reflection group preserving $M$, with simple system $\Delta = \{\alpha_1, \ldots, \alpha_r\} \subset M$ as in \cref{goodroots},
\item $\lambda_1, \ldots, \lambda_r$ the corresponding fundamental weights,
\item $D = Z \times \R_{\geq 0}\gen{\lambda_1, \ldots, \lambda_r}$ the corresponding fundamental domain, and
\item $\Lambda = Z \oplus \Z\gen{\lambda_1, \ldots, \lambda_r} \supset M$ the weight group, and $\Lambda_\circ = \Z\gen{\lambda_1, \ldots, \lambda_r}$.
\end{itemize}
All future mention of ``simple roots'' will refer to elements of $\Delta$.
If $S$ is a semigroup, let $\Z[S]$ denote its semigroup algebra, i.e. the group of finite $\Z$-linear combinations of symbols $\{\chi^s : s \in S\}$ with multiplication given by $\chi^s \chi^t = \chi^{s+t}$.

\subsection{} The decompositions in \cref{groups:crystallographic} imply that
\begin{align*}
    \Z[D \cap \Lambda] &\cong \Z[Z] \otimes_\Z \Z[\Lambda_\circ] = \Z[Z] \otimes_\Z \Z[\chi^{\lambda_1}, \ldots, \chi^{\lambda_r}], \\
    \Z[\Lambda]^W &\cong \Z[Z] \otimes_\Z \Z[\Lambda_\circ]^W = \Z[Z] \otimes_\Z \Z[\chi^{\lambda_1},\ldots, \chi^{\lambda_r}]^W,
\end{align*}
and that the generators $\chi^{\lambda_1}, \ldots, \chi^{\lambda_r}$ are algebraically independent.
Consequently, there is a well-defined map of $\Z$-algebras
\[
  \Psi_\circ: \Z[D \cap \Lambda_\circ] \to \Z[\Lambda_\circ]^W
\]
determined by
\[
  \Psi_\circ(\chi^{\lambda_i}) = \sum_{v \in W\lambda_i} \chi^v, \quad 1 \leq i \leq r.
\]
This map extends uniquely to a map of $\Z[Z]$-algebras
\[
  \Psi: \Z[D \cap \Lambda] \to \Z[\Lambda]^W.
\]
\begin{theorem}\cite[Theorem VI.3.4.1]{B02}\label{bourbaki}
  $\Psi_\circ$ is an isomorphism.
\end{theorem}
 \cref{bourbaki} implies $\Psi$ is also an isomorphism.
Since our proof of \cref{main} already requires analysis of $\Psi$, we provide a proof of \cref{bourbaki} in \cref{proofs:main}.
\subsection{}\label{Msupport} The elements
\[
  \underline\chi^u = \sum_{v \in Wu} \chi^v, \quad u \in D \cap \Lambda
\]
are a $\Z$-basis of $\Z[\Lambda]^W$, and the \word{support} of $f  = \sum_{u \in D \cap \Lambda} a_u \underline\chi^u$ is
\[
  \supp(f) := \{ u \in D \cap \Lambda : a_u \neq 0 \}.
\]
We study the supports of elements $\Psi(\chi^u)$.

\begin{lemma}\label{psisupport}
  Let $u \in D \cap \Lambda_\circ$.
  \begin{enumerate}
  \item If $v \in \supp \Psi_\circ(\chi^u)$, then $u - v \in \Z_{\geq 0}\gen{\alpha_1, \ldots, \alpha_r}$. In particular, $u \geq v$.
  \item \label{psisupport:1} $\underline \chi^u$ has coefficient 1 in $\Psi_\circ(\chi^u)$.
  \end{enumerate}
\end{lemma}
\begin{proof}
  Write $u = \sum_{i=1}^r b_i \lambda_i$ with $b_1, \ldots, b_r \in \Z_{\geq 0}$.
  We may expand
  \[
    \Psi(\chi^u) = \Psi_\circ(\chi^u) =  \prod_{i=1}^r \left(\sum_{v \in W\lambda_i} \chi^v\right)^{b_i}
    =  \sum_{v \in D \cap \Lambda_\circ} a_v \underline \chi^v,
  \]
  where
  \[
    a_v = \#\left\{ (v_{ij}) \in \prod_{\substack{i=1, \ldots, r\\j = 1, \ldots, b_i}} W \lambda_i : \sum_{ij} v_{ij} = v \right\}.
  \]
  If $v \in \supp \Psi(\chi^u)$, then
  \[
    v = \sum_{\substack{i=1, \ldots, r\\j = 1, \ldots, b_i}} v_{ij}
  \]
  for some $(v_{ij}) \in \prod_{\substack{i=1, \ldots, r\\j = 1, \ldots, b_i}} W \lambda_i$.
  By \cref{groups:ordering:orbit}, $v_{ij} \leq \lambda_i$ for all $i$, so
  \[
        v = \sum_{\substack{i=1, \ldots, r\\j = 1, \ldots, b_i}} v_{ij} \leq \sum_{\substack{i=1, \ldots, r\\j = 1, \ldots, b_i}} \lambda_i = u.
  \]
  This inequality is an equality if and only if $v_{ij} = \lambda_i$ for all $i,j$, so $a_u = 1$.
  Finally, $u - v \in \Z \gen{\alpha_1, \ldots, \alpha_r}$ because $\lambda_i - v_{ij} \in \Z\gen{\alpha_1, \ldots, \alpha_r}$ for all $i,j$.
\end{proof}
\begin{lemma}\label{restriction}
  If $S \subset M$ is a saturated affine semigroup preserved by the action of $W$, and $u \in D \cap S$, then 
  $\supp \Psi(\chi^u) \subset S$.
\end{lemma}
\begin{proof}
  Let $C := \R_{\geq 0} \gen{s : s \in S} \subset V$.
  Plainly, $C$ is convex and preserved by $W$. Moreover, $S = C \cap M$ by  \cite[Theorem 1.3.5]{CLS11}.

  Let $u \in D \cap S$ and $v \in \supp \Psi(\chi^u)$.
  We may write uniquely $u = z + u_\circ$ and $v = z + v_\circ$, with $v_\circ \in \supp \Psi_\circ(\chi^{u_\circ})$.

  We first show that $v \in M$.
  By \cref{psisupport}, $u - v = u_\circ - v_\circ \in \Z_{\geq 0}\gen{\alpha_1, \ldots, \alpha_r}$.
  Since $u, \alpha_1, \ldots, \alpha_r \in M$, $v$ must also lie in $M$.

  Next, we show that $v \in C$.
  By \cref{psisupport}, $u_\circ \geq v_\circ$.
  By \cref{groups:ordering:convex}, this means $v_\circ$ is in the convex hull of $W u_\circ$.
  Consequently,
  \[
    C \supset \conv(Wu) = z + \conv(Wu_\circ) \ni z + v_\circ = v.
  \]
  We conclude that $v \in C \cap M = S$.
\end{proof}
\subsection{}\label{proofs:main}
We now move to establish \cref{main,main:affine}, and \cref{main:projective}.
\begin{lemma}\label{linindep}
  Let $R$ be any commutative ring.
  If $U \subset D \cap \Lambda_\circ$ is a finite set of weights, then $\{\Psi_\circ(\chi^u) : u \in U\} \subset R \otimes_\Z \Z[\Lambda_\circ]$ is linearly independent over $R$.
\end{lemma}
\begin{proof}
  We induct on the cardinality of $U$.
      If $\#U = 1$, the statement is true by \cref{psisupport}\cref{psisupport:1}.

  Otherwise, if $\#U > 1$, let $v \in U$ be a maximal element with respect to $\leq$.
  By \cref{psisupport}, there is no $v' \in U \setminus v$ such that $v \in \supp \Psi_\circ(\chi^{v'})$; however, $v \in \supp \Psi_\circ(\chi^v)$. Consequently, $\Psi_\circ(\chi^v)$ is not in the span of $\{\Psi_\circ(\chi^u) : u \in U \setminus v\}$. By the induction hypothesis, it follows that $\{\Psi_\circ(\chi^u) : u \in U\}$ is linearly independent.
\end{proof}

\begin{proof}[Proof of \cref{bourbaki}]
  We first show $\Psi_\circ$ is injective.
  If $f \in \ker(\Psi_\circ)$, then $f$ is contained in a subgroup of the form $\Z\gen{U} := \Z\gen{\chi^u : u \in U}$ with $U \subset D \cap \Lambda_\circ$ a finite set.
  \cref{linindep} implies that $\Psi_\circ$ is injective on $U$, so $f = 0$.

  Next, we show that $\Psi_\circ$ is surjective.
  To do this, we will induct to show that $\underline\chi^u$ is in the image of $\Psi_\circ$ for each $u \in D \cap \Lambda_\circ$.
  Given $u \in D \cap \Lambda_\circ$, let
  \[ n(u) := \#\{v \in D \cap \Lambda_\circ : v \leq u \}. \]
  We have $n(u) = 1$ if and only if $u$ is a minimal element of $D \cap \Lambda_\circ$ with respect to $\leq$ (minimal elements exist because for any $u \in D \cap \Lambda_\circ$,  $\{v \in D \cap \Lambda_\circ : v \leq u\}$ is the intersection of the discrete set $\Lambda_\circ$ with a compact set by \cref{groups:ordering:convex}, hence finite).

  If $n(u) = 1$, then \cref{psisupport} implies $\Psi_\circ(\chi^u) = \sum_{v \in Wu} \chi^v = \underline \chi^u$.
  Otherwise, suppose $n(u) > 1$. By the induction hypothesis and \cref{psisupport},
  \[
    \Psi_\circ(\chi^u) - \underline\chi^u \in \Z\gen{\underline \chi^v : v < u} \subset \Z\gen{\underline \chi^v : n(v) < n(u)} \subset \Psi_\circ(\Z[D \cap \Lambda_\circ]),
  \]
  so $\underline \chi^u$ is in the image of $\Psi_\circ$.
  This completes the proof.
\end{proof}
\begin{proof}[Proof of \cref{main}]
  By \cref{bourbaki}, $\Psi_\circ$ is an isomorphism. Since $\Z[Z]$ is free over $\Z$, it follows that $\Psi$ is also an isomorphism. It now suffices to show that $\Psi(\Z[D \cap S]) = \Z[S]^W$.
  
  The containment $\Psi(\Z[D \cap S]) \subset \Z[S]^W$ holds by \cref{restriction}.
  For the reverse containment, suppose $f \in \Z[D \cap \Lambda]$ is not in $\Z[D \cap S]$.
  There there is a maximal $v \in D \cap \Lambda$ such that $v \not \in S$, but $\chi^v$ appears in $f$.
  By \cref{psisupport} and \cref{restriction}, if $\chi^{v'}$ appears in $f$ with $v' \neq v$, then $v \not \in \supp \Psi(\chi^{v'})$; therefore, $v \in \supp \Psi(f)$.
  It follows that $\Psi(f)$ is not contained in $\Z\gen{\chi^s : s \in S}$, hence is also not contained in
  \[
    \Z[S]^W = \Z\gen{\chi^s : s \in S} \cap \Z[\Lambda]^W.
  \]
  We conclude that $\Psi(f) \in \Z[S]^W$ if and only if $f \in \Z[D \cap S]$.
\end{proof}
\begin{proof}[Proof of \cref{main:affine}]
  In light of \cite[Theorem 1.3.5]{CLS11}, we need only check that the linear functionals $f_\alpha = \gen{\alpha, -}$, $\alpha \in \Delta$, defining $D$ can be rescaled to lie in $\Hom(M, \Z)$.
  This is possible because $M \subset \Lambda$, so $\frac{2 f_\alpha}{\ip{\alpha,\alpha}}$ takes integer values on $M$.
\end{proof}
\begin{proof}[Proof of \cref{main:projective}]
  Recall from \cite[Chapter 2]{CLS11} that an $M$-rational polytope $P \subset V$,  gives rise to the projective toric variety $X_P := \Proj \Z[S_P]$, where
  \[
    S_P := \{ (t, p) : t \in \Z_{\geq 0}, p \in tP \} \subset \Z \times M,
  \]
  graded by $\deg (t,p) = t$.
  Since $W$ preserves $M$ and $P$, it acts on $S_P$, hence on $\Z[S_P]$, yielding an action on $X_P$ by toric automorhpisms.
  The fundamental domain of the $W$-action on $\R \times V$ is $\R \times D$, so by \cref{main}, there is  an isomorphism of graded rings
  \[ \Z[S_{D \cap P}] = \Z[(\R \times D) \cap S_P] \to \Z[S_P]^W. \]
  Applying $\Proj$, we obtain an isomorphism $X_P / W \cong X_{D \cap P}$.
\end{proof}

\begin{example}\label{quotient-example}
  Suppose that $W = S_2$ acts on $M = \Z^2$ by coordinate 
  permutation, and $S = \left(\Z_{\geq 0}\right)^2 \subseteq M$. Then 
  $W$ is generated by reflection across the single simple 
  root $\alpha = (-1,1)$, with corresponding weight group $\Lambda = Z + 
  \Z\cdot \lambda$, where $Z$ is the real span of $z = \frac{e_1 + 
  e_2}{2}$ and $\lambda = \frac{-e_1 + e_2}{2}$. The map $\Psi:\Z[D \cap 
  S] \to \Z[S]^W$ is given by 
  \[
    \Psi\left(\chi^{az + b\lambda}\right) = 
          (\chi^{z})^a\left(\chi^{\lambda} + 
            \chi^{s_{\alpha}(\lambda)}\right)^b,\quad
          az + b\lambda \in D \cap S.
  \]
  For instance, if $u = e_1 + 3e_2 = 4z + 2\lambda$, then $\Psi(\chi^u) = 
  \chi^{4z} \left(\chi^{\lambda} + \chi^{s_{\alpha}(\lambda)}\right)^2 
  = \chi^{4z + 2\lambda} + 2 \chi^{4z} + \chi^{4z - 2\lambda}$. See \cref{one-reflection-picture}. 
  Geometrically, we have that $\A^2/{S_2} \cong \A^2$.
  
   \begin{figure}[!htb]
  \begin{center}
   \begin{tikzpicture}
\draw[<->]   (0,6) |- (3,3);
\draw[->, very thick] (0,3) -- (-1,4) node[above left]{$\alpha$};
\draw[->, very thick] (0,3) -- (-0.5,3.5) node[above right]{$\lambda$};
\draw[->, very thick] (0,3) -- (0.5,2.5) node[right]{$s_{\alpha}(\lambda)$};
\draw[->, very thick] (0,3) -- (0.5,3.5) node[right]{$z$};
\foreach \x [count=\xi] in {-3,...,3}
\foreach \y [count=\yi from 0] in {-3,...,3}
\fill   (\x,\yi) circle[radius=2pt];

\fill[semitransparent, lightgray] (0,3) -- (0,6) -- (3,6);

\fill[color=black] (1,6) circle[radius=4pt];
\node[below right] at (1,6) {$\times 1$};

\fill[color=black] (2,5) circle[radius=4pt];
\node[below right] at (2,5) {$\times 2$};

\fill[color=black] (3,4) circle[radius=4pt];
\node[below right] at (3,4) {$\times 1$};

\end{tikzpicture} 
\caption{\label{one-reflection-picture} The lattice points in the shaded 
  cone are $D \cap S$, and the bold lattice points marked with coefficients represent 
$\Psi(\chi^{u})$ for $u = e_1 + 3e_2$.}
 \end{center} 
\end{figure}

\end{example}
\begin{remark}[Sublattices of finite index]
    A finite-index inclusion $M \hookrightarrow M'$ of lattices induces torus-equivariant maps $\pi: X_{P, M'} \to  X_{P, M}$ and $\pi': X_{P \cap D, M'} \to X_{P\cap D, M}$ \cite[Proposition 3.3.7]{CLS11}.
    The dense tori of $X_{P, M'}$ and $X_{P\cap D, M'}$ are canonically identified on account of having the same character lattice.
    Under this identification, the restrictions of $\pi$ and $\pi'$ to the dense tori of $X_{P, M'}$ and $X_{P \cap D,M'}$ are equal; both $\pi$ and $\pi'$ are quotients by their common kernel $G$.

    If $M$, $M'$, and $P$ are all preserved by $W$, then $G$ is fixed pointwise by $W$ because $M$ contains the root lattice.
  Consequently, there is a commuting square of quotient maps
  \[\begin{tikzcd}
      X_{P,M'} \ar[r] \ar[d, "\pi"] & X_{P,M'} / W \cong X_{P \cap D, M'} \ar[d, "\pi'"]  \\
      X_{P, M} \ar[r] & X_{P,M}/W \cong X_{P \cap D, M}.
  \end{tikzcd}\]
\end{remark}

\section{Quotients of real points}\label{sec:realpts}
We review topological models for $X_P^\R$ and $X_P^\R / W$ in \cref{sec:realptsmodel} and \cref{sec:quotientmodel}, respectively. In  \cref{sec:realpts:surface}, we prove \cref{quotients of real surfaces}, which describes the homotopy type of a quotient of a real toric surface.

Except where otherwise noted, we follow the notations of \cref{proofs}, with the following additions:
\begin{itemize}
  \item the simple roots $\Delta$ are assumed to span $V$, so $r = \#\Delta = \dim V = \dim P$,
  \item $\Delta^\vee$ is the set of simple coroots $\alpha^\vee = \frac{2\alpha}{\ip{\alpha,\alpha}}$ for $\alpha \in \Delta$,
  \item $\Lambda^\vee$ is the coweight lattice, with basis $\lambda_1^\vee, \ldots, \lambda_r^\vee$ dual to the simple roots,
  \item $N = \Hom(M, \Z)$ is the dual of $M$,
  \item $\mathcal S = \Hom(M, \{\pm 1\}) \cong N / 2N$ is the \word{sign subgroup} of the dense torus of $X_P$, and
  \item $Q^\circ$ is the relative interior of a polyhedron $Q$.
\end{itemize}
Despite the fact that $\mathcal S$ is a subgroup of the dense torus of $X_P$, we will write it additively via the isomorphism $\mathcal S \cong N/2N$.
The exposition of \cref{sec:realptsmodel} loosely follows \cite{soprunovaLowerBoundsReal2013}; readers are referred to \cite{CLS11} for any undefined terms pertaining to toric or polyhedral geometry.

\subsection{}\label{sec:realptsmodel}
Let $X_P^\R$ denote the real points of $X_P$.
The sign group $\mathcal S$ is contained in the real points of the dense torus in $X_P$, hence acts on $X_P^\R$.
If $p \in X_P^\R$ belongs to the torus orbit $O_F$ corresponding to a face $F$ of $P$, then the stabilizer in $\mathcal S$ of $p$ is the subgroup $\mathcal S_p \subset \mathcal S$ generated by lattice points in the inner normal cone to $F$.

Each $\mathcal S$-orbit of $X_P^\R$ intersects the non-negative part $X_P^\geq$ of $X_P$ in a unique point. The toric moment map $W$-equivariantly identifies $X_P^\geq$ with $P$, sending the non-negative part of $O_F$ homeomorphically onto $F^\circ$. Hence, we obtain an oft-rederived topological model for $X_P^\R$ (see, e.g. \cite[\S2.4]{jurkiewicz1985torus}, \cite[\S2]{soprunovaLowerBoundsReal2013}, \cite[\S2]{G24}).

\begin{theorem}\label{realptsmodel}
  The space $X_P^\R$ is homeomorphic to
  \[
\mathcal S \times P / \sim,
  \]
  where $\epsilon \times p \sim \delta \times q$ if and only if $p=q$ and $\epsilon-\delta \in \mathcal S_p$.
\end{theorem}
\begin{corollary}
 The space $X_P^\R$ is a connected CW complex.
\end{corollary}
\begin{proof}
  The topological description  shows that $X_P^\R$ is obtained by gluing copies of $P$ along faces; the faces give a cell decomposition.
  For connectedness, observe that if $p$ is a vertex of $P$, then for all $\epsilon, \epsilon' \in \mathcal S$,  $\epsilon \times p \sim \epsilon' \times p$.
\end{proof}

\subsection{}\label{sec:quotientmodel}
Restricting \cite[Lemma 4.4]{G24} to real points, we also obtain a topological description for $X_P^\R/W$. Since the setup for \cite[Lemma 4.4]{G24} differs slightly from the present one, we give a proof.
\begin{theorem}\label{quotientmodel}
  The quotient space $X_P^\R / W$ is a connected CW complex homeomorphic to
  \[
\mathcal S  \times (P \cap D) / \sim_W,
  \]
  where $\epsilon \times p \sim_W \delta \times q$ if and only if $p = q$ and there is $w \in W_p := \{ w \in W : w(p) = p\}$ such that $w(\epsilon) - \delta \in \mathcal S_p$.
\end{theorem}
\begin{proof}
  The action of $W$ preserves both $X_P^\geq$ and $\mathcal S$, and satisfies $w(\epsilon p) = w(\epsilon) w(p)$ for all $w \in W$, $\epsilon \in \mathcal S$, and $p \in X_P^\geq$.
  Consequently, every fiber of the composition
  \[
\mathcal S \times P \twoheadrightarrow X_P^\R \twoheadrightarrow X_P^\R / W
  \]
  contains a point of $\mathcal S \times (P \cap D)$, and we obtain a surjection
  \[
    \mathcal S \times (P \cap D) \twoheadrightarrow X_P^\R / W.
  \]
  Two points $\epsilon \times p$ and $\delta \times q$ have the same image if and only if there is $w \in W$ such that
  \begin{equation*}
    q = (-\delta)(w(\epsilon p)) = (w(\epsilon) - \delta) w(p).
  \end{equation*}
  The conditions $p=q$, $w(p) = p$, and $w(\epsilon) - \delta \in \mathcal S_p$ plainly suffice to make this equation hold.
  Conversely, suppose the equation holds.
  Since $w(p)$ and $q$ are points of $X_P^\geq$ in the same $\mathcal S$-orbit, $w(p) = q$.
  In turn, $p$ and $q$ are points of $P \cap D$ in the same $W$-orbit, so $p = q$.
  It follows that $w(p) = p$ and $w(\epsilon) - \delta \in \mathcal S_p$.
\end{proof}

\subsection{}\label{sec:realpts:surface}
In the remainder of this section, we prove \cref{quotients of real surfaces}, our result on real toric surfaces.
We prepare for the proof with a few lemmas.
\begin{lemma}\label{r-1lem}
  If for all but at most one $\epsilon \in \mathcal S$ there is $\alpha \in \Delta$ such that $s_\alpha(\epsilon) = \epsilon$, then $X_P^\R / W$ is homotopy equivalent to a CW complex with cells of dimension at most $r-1$.
\end{lemma}
\begin{proof}
  By \cref{quotientmodel}, $X_P^\R / W$ is constructed by gluing together $2^r$ copies of $P \cap D$ along faces.
  If $\epsilon$ is stabilized by a simple reflection $s_\alpha$, then no point of $\epsilon \times (\alpha^\perp \cap P \cap D)^\circ$ is glued. Therefore, $\epsilon \times (P \cap D)$ can be retracted onto a union of faces in its boundary.

  After performing such retractions for all but one copy of $P \cap D$, then contracting the final copy of $P \cap D$ to a point, we are left with a CW complex homotopy equivalent to $X_P^\R / W$ that has at most $(r-1)$-dimensional cells.
\end{proof}
\begin{lemma}\label{atmost1:rootlattice}
  If $M$ is the root lattice, then all but at most one element of $\mathcal S$ is fixed by some simple reflection.
\end{lemma}
\begin{proof}
  In this case, $N$ is the coweight lattice $\Lambda^\vee$. Using the isomorphism $\mathcal S \cong N /2N =  \Lambda^\vee / 2 \Lambda^\vee$, express $\epsilon = c_1 \lambda_1^\vee + \cdots + c_r \lambda_r^\vee + 2 \Lambda^\vee$ as a linear combination of fundamental coweights. If $c_i = 0$, then $s_{\alpha_i}$ fixes $\epsilon$, so the unique element that might be fixed by no simple reflection is $\sum_i \lambda_i^\vee + 2\Lambda^\vee$.
\end{proof}
\begin{lemma}\label{atmost1:other}
  If $W$ is irreducible of type $A_{2n}$, $C_n$, $E_6$, $E_7$, $E_8$, $F_4$, or $G_2$, then all but at most one element of $\mathcal S$ is fixed by some simple reflection.
\end{lemma}
\begin{proof}
  Recall that we have chosen simple roots so that $\Delta \subset M \subset \Lambda$, equivalently, so that $\Delta^\vee \subset N \subset \Lambda^\vee$.
  Suppose that $u+2N \in \mathcal S \cong N/2N$ is fixed by no simple reflection. Then $u = \sum_{i=1}^r c_i \lambda_i^\vee$ with all coefficients $c_i$ odd (otherwise, if $c_i$ were even, then $u - s_{\alpha_i}(u) = c_i\alpha_i^\vee \in 2N$, contrary to the assumption that $u+2N$ is fixed by no simple reflection). Hence, $u$ is a preimage for $\epsilon_0 := \sum_{i=1}^r \lambda_i^\vee + 2 \Lambda^\vee$ under the map $N/2N \to \Lambda^\vee / 2\Lambda^\vee$. It now suffices to show in each case that $\epsilon_0$ has at most one preimage.

  In the cases $A_{2n}$, $E_6$, $E_8$, $F_4$, or $G_2$, $N$ must have odd index in $\Lambda^\vee$, so the map $\mathcal S \cong N / 2N \to \Lambda^\vee / 2\Lambda^\vee$ is an isomorphism.

  In types $C_n$ and $E_7$, the coroot lattice has index 2 in the coweight lattice. Accounting for \cref{atmost1:rootlattice}, we need only consider when $N$ is the coroot lattice.
  In this case, $N / 2N \to \Lambda^\vee / 2\Lambda^\vee$ is given by the transpose of the Cartan matrix mod 2, and it is straightforward to verify that $\epsilon_0$ is not in the image (for pictures of the Cartan matrices, see \cite[pp. 547 \& 553]{carter}).
\end{proof}
\begin{proof}[Proof of \cref{quotients of real surfaces}]
  By Lemmas \ref{r-1lem}, \ref{atmost1:rootlattice}, and \ref{atmost1:other}, $X_P^\R / W$ is homotopy equivalent to a wedge of circles.
  We discover how many by computing the Euler characteristic.

  Let $k$ be the number of vertices of $P$ that lie in $D^\circ$.
  The polyhedron $P \cap D$ has $k+3$ vertices: $k$ vertices in $D^\circ$, the origin, and two vertices on the walls of $D$, which we term \word{corners}.
  The following facts about cell counts hold:
  \begin{itemize}
    \item $X_P^\R / W$ has 4 two-cells, indexed by elements of $\mathcal S$.
    \item Each of the $k+1$ edges of $\epsilon \times (P \cap D)$ in $D^\circ$ is glued to exactly one copy of the same edge, giving $2(k+1)$ one-cells. The edges $0 \times (\alpha_i^\perp \cap P \cap D)$ are unglued, giving an additional 2 one-cells. This yields a partial count of $2(k+1)+2$ one-cells, with up to 6 additional one-cells, whose count we shall determine on a case-by-case basis.
    \item Each of the $k$ vertices of $\epsilon \times (P \cap D)$ in $D^\circ$ is glued to all other copies of the same vertex, giving $k$ zero-cells. The vertex $0 \times 0$ is unglued, and the corners of $0 \times (P \cap D)$ are each glued to exactly one other copy of themselves. All corners are glued to at least one other copy of themselves. This yields a partial count of $k+3$ zero-cells, with up to 6 additional zero-cells, whose count we shall determine on a case-by-case basis.
  \end{itemize}

  We now determine the counts of ``additional'' cells case-by-case.
  There are three possible positions for $P$ relative to $D$, depicted in \cref{three shapes}; assume that $P$ is of Shape I. In this case, $k$ is equal to the number $\ell$ of vertices of $P$ that lie in $D$.
  \smallskip

    In types $A_2$ and $G_2$, $N$ has odd index in $\Lambda^\vee$, so the map $\mathcal S \cong N/2N \to \Lambda^\vee / 2\Lambda^\vee$ is an isomorphism.
  Hence,  we may assume without loss of generality that $M$ is the root lattice and $N = \Lambda^\vee$.
  From the Cartan matrices of $A_2$ and $G_2$ taken mod 2, we see that
  \begin{itemize}
    \item The edge $(\lambda_i^\vee +2N) \times (\alpha_i^\perp \cap P \cap D)$ is glued to $(\lambda_1^\vee+\lambda_2^\vee +2N) \times (\alpha_i^\perp \cap P \cap D)$ for $i = 1,2$, and no further edge gluings occur. There is a total of $2(k+1)+2+4 = 2k+8$ one-cells.
    \item The vertices $\epsilon \times 0$ for $\epsilon \neq 0$ are all glued together, and the four remaining vertices are glued in two pairs. There is a total of $k+3+1 + 2 = k+6$ zero-cells. 
  \end{itemize}
  These counts give an Euler characteristic of $2-k$.

  In type $C_2$, the root lattice has index 2 in the weight lattice, so $M$ is either the root lattice or the weight lattice.
  When $M$ is the root lattice,
  $(\lambda_2^\vee+2N) \times (\alpha_2^\perp \cap P \cap D)$ is glued to $(\lambda_1^\vee + \lambda_2^\vee+2N) \times (\alpha_2^\perp \cap P \cap D)$. This edge gluing forces two vertex gluings. In total, we have 4 two-cells, $2(k+1)+2+5$ one-cells, and $k+3+4$ zero-cells, for an Euler characteristic of $2-k$.

  Finally, suppose $M$ is the weight lattice of type $C_2$ (so $N$ is the coroot lattice).
  The action of $W$ on $N$ fixes $\alpha_1^\vee$, and $s_{\alpha_1}(\alpha_2^\vee + 2N) = \alpha_2^\vee + \alpha_1^\vee + 2N$. This creates a single edge gluging; the final cell counts are the same as when $M$ is the root lattice, for an Euler characteristic of $2-k$.

  The proof is now complete when $P$ is of Shape I.
  For Shapes II and III, the counts of two-cells and one-cells are identical; for zero-cells, the only difference is at the corners of $P \cap D$.
  If a corner is a vertex of $P$, then all copies of it are identified to form a single zero-cell in $X_P^\R/W$ (rather than the 2 zero-cells it would form otherwise), decreasing Euler characteristic by 1. Correspondingly, the number $\ell$ of vertices of $P$ that lie in $D$ is increased by 1, and we obtain
    $X_P^\R / W \cong \bigvee_{\ell-1} S^1$, as desired.
\end{proof}

\begin{figure}[htpb]
  \centering
  \begin{tikzpicture}
    \begin{scope}[scale=0.7,shift={(-7,0)}]
      \newdimen\r
      \r=3cm
      \draw (75:\r)--(105:\r);
      \draw (15:\r)--(45:\r);
      \draw[dotted] (45:\r)--(75:\r);
      \fill (75:\r) circle (3pt);
      \fill (105:\r) circle (3pt);
      \fill (15:\r) circle (3pt);
      \fill (45:\r) circle (3pt);
      \draw[blue,->] (0:0) -- (30:{1.25*\r}) node[above] {$\lambda_1$};
      \draw[blue,->] (0:0) -- (90:{1.25*\r}) node[left] {$\lambda_2$};
      \node at (270:{0.5*\r}) {Shape I};
    \end{scope}
    \begin{scope}[scale=0.7]
      \newdimen\r
      \r=3cm
      \draw (75:\r)--(105:\r);
      \draw (15:\r)--(45:\r);
      \draw[dotted] (45:\r)--(75:\r);
      \draw (15:\r)--(0:{cos(15)*\r});
      \fill (75:\r) circle (3pt);
      \fill (105:\r) circle (3pt);
      \fill (15:\r) circle (3pt);
      \fill (45:\r) circle (3pt);

      \draw[blue,->] (0:0) -- (15:{1.25*\r}) node[above] {$\lambda_1$};
      \draw[blue,->] (0:0) -- (90:{1.25*\r}) node[left] {$\lambda_2$};
      \node at (270:{0.5*\r}) {Shape II};
    \end{scope}
    \begin{scope}[scale=0.7,shift={(7,0)}]
      \newdimen\r
      \r=3cm
      \draw (75:\r)--(90:{cos(15)*\r});
      \draw (15:\r)--(30:{cos(15)*\r});
      \draw[dotted] (30:{cos(15)*\r})--(60:{cos(15)*\r});
      \draw (60:{cos(15)*\r})--(75:\r);
      \draw (15:\r)--(0:{cos(15)*\r});
      \fill (75:\r) circle (3pt);
      \fill (15:\r) circle (3pt);

      \draw[blue,->] (0:0) -- (15:{1.25*\r}) node[above] {$\lambda_1$};
      \draw[blue,->] (0:0) -- (75:{1.25*\r}) node[left] {$\lambda_2$};
      \node at (270:{0.5*\r}) {Shape III};
    \end{scope}
  \end{tikzpicture}
  \caption{Three possible shapes of the quotient polygon $P \cap D$. The blue arrows indicate the fundamental weights $\lambda_1,\lambda_2$ that define the fundamental domain $D$.}\label{three shapes}
\end{figure}

\section{Quotients of real permutohedral varieties}\label{Higher dimensional cases}
We establish \cref{main-thm-real}. The major facts are laid out in \cref{realstrategy}; the remainder of the section is devoted to proving the key \cref{gluedspacelemma}.
Throughout, we follow the conventions of \cref{sec:realpts} with the additional assumptions:
\begin{itemize}
  \item $W$ is irreducible,
  \item $M$ is the root lattice of $W$, hence $N = \Lambda^\vee$ is the coweight lattice,
  \item $P$ is a \word{permutohedron}, meaning that it is the convex hull of the $W$-orbit of a point in $D^\circ$,
  \item $s_i := s_{\alpha_i}$ is the simple reflection corresponding to the simple root $\alpha_i$, and
  \item $[r]:=\{1,2,\ldots,r\}$.
\end{itemize}
Via the basis of fundamental coweights, we identify $\mathcal S \cong N / 2N \cong (\Z/2)^r$.
We will write $\lambda_j^\vee$ for both the fundamental coweight, and for its  image in $N/2N$, trusting that this will cause no confusion. For $u\in\mathcal{S}$, we also write $u_i$ for its coefficient of $\lambda_i^{\vee}$.
\subsection{}\label{realstrategy}
We explain our approach to \cref{main-thm-real}.
Order $\mathcal S \cong (\Z/2)^r$ lexicographically, i.e.\ $(u_1, \ldots, u_r) > (v_1, \ldots, v_r)$ if $u_i = 1$ and $v_i = 0$ for $i = \min\{j : u_j \neq v_j\}$.
Per \cref{quotientmodel}, $X_P^\R/W$ is constructed by gluing together copies of $P \cap D$ indexed by $\mathcal S$.
For each $u \in \mathcal S$, there is an inclusion $\pi_u: u \times (P \cap D) \hookrightarrow X_P^\R / W$ and a pushout diagram
\[\begin{tikzcd}
     \mathbf{G}_u :=\pi_u^{-1}\left(\underset{v>u}{\bigcup} \im(\pi_v)\right) \arrow[d,hookrightarrow,"f_u"']\arrow[r,"\pi_{u}"]\arrow[rd,phantom,"\ulcorner",very near end] & \underset{v > u}{\bigcup}\im(\pi_v)\arrow[d,hookrightarrow,"g_u"]\\
     u \times\left(P\cap D\right)\arrow[r,"\pi_{u}"] & \mathop{\bigcup}\limits_{v \geq u}\im(\pi_v),
   \end{tikzcd}\]
 (in which the ``$\mathbf{G}$'' in ``$\mathbf{G}_u$'' is short for ``Glued Subspace''.)
The linchpin of our proof is:
\begin{lemma}\label{gluedspacelemma}
  The space $\mathbf{G}_u$ is contractible for all $u < \sum_{i=1}^r \lambda_i^\vee$.
 \end{lemma}
 Assuming \cref{gluedspacelemma}, \cref{main-thm-real} is straightforward.
 \begin{proof}[Proof of \cref{main-thm-real}]
   All glued spaces $\mathbf{G}_u$ are contractible by \cref{gluedspacelemma}, so all the maps $f_u$ are homotopy equivalences.
   Consequently, the pushouts $g_u$ are also homotopy equivalences.
   By composing the $g_u$'s in decreasing order of $u$, we obtain a homotopy equivalence
 \begin{align*}
   (1,1,\ldots,1) \times (P \cap D) {\longrightarrow} \cdots \overset{g_u}{\longrightarrow} \bigcup_{v \geq u} \im(\pi_v) \longrightarrow  \cdots \overset{g_0}{\longrightarrow} \bigcup_{v \geq 0} \im(\pi_v) = X_P^\R / W,
 \end{align*}
 so $X_P^\R / W$ is contractible because $P \cap D$ is.
\end{proof}
The remainder of this section is devoted to proving \cref{gluedspacelemma}.

\subsection{}\label{sec:faces}
We first describe the face structure of $P \cap D$, following \cite[\S 3]{G24} (see also \cite{burrull2024stronglydominantweightpolytopes}).
The polytope $P \cap D$ is combinatorially equivalent to a cube $[0,1]^r \subset \R^r$.
For each pair of disjoint subsets $I, J \subset [r] := \{1, \ldots, r\}$, there is a face $H_IF_J$ of dimension $r- \#I-\#J$ with inner normal cone generated by
\[
  \{\alpha_{i} :  i \in I\} \cup \{-\lambda_j^\vee : j \in J\}.
\]
Two faces $H_IF_J$ and $H_{I'}F_{J'}$ are disjoint if and only if $(I \cap J') \cup (J \cap I') \neq \emptyset$. Otherwise, their intersection is $H_IF_J \cap H_{I'}F_{J'} = H_{I \cup I'}F_{J \cup J'}$. We will often write $H_I$ and $F_J$ for the faces $H_\emptyset F_J$ and $H_I F_\emptyset$, respectively.

If $U \subset P \cap D$ is a union of faces, then there is a unique decomposition
\[
  U = H_{I_1}F_{J_1} \cup H_{I_2}F_{J_2} \cup \cdots \cup H_{I_k}F_{J_k},
\]
such that no component $H_{I_j}F_{J_j}$ in the union is contained in any other.
The components of this union are the \word{irreducible components} of $U$.
To establish \cref{gluedspacelemma}, we will employ the following criterion for contractibility.

\begin{proposition}\label{not_opposite_component}
  Let $U \subset P \cap D$ be a union of faces.
  If there is an irreducible component of $U$ that intersects all other irreducible components of $U$, then $U$ is contractible.
\end{proposition}
\begin{proof}
  Identify $P \cap D \cong [0,1]^r$ so that each face $H_I F_J$ is identified with $\{0\}^I \times \{1\}^J \times [0,1]^{[r] \setminus (I \cup J)}$.
  The map
  \[
    [0,1] \times \R^I \times \R^J \times \R^{[r] \setminus (I \cup J)} \to \R^{I}\times \R^{J} \times \R^{[r] \setminus (I \cup J)}, \quad (t,x,y,z) \mapsto (tx, (1-t)(1,1,\ldots,1) + ty, z)
  \]
  gives a deformation retraction of $[0,1]^r$ onto $H_I F_J$, and on each face $H_{I'}F_{J'}$ that intersects $H_I F_J$, it restricts to a retraction of $H_{I'}F_{J'}$ onto $H_{I'}F_{J'} \cap H_I F_J$. Consequently, if $H_I F_J$ is an irreducible component of $U$ that intersects all others, then $U$ is homotopy equivalent to the contractible space $H_I F_J$.
\end{proof}
\subsection{}\label{words}
\newcommand{\SR}{\operatorname{Eff}}
\newcommand{\sr}{\SR}
We develop some combinatorial notions that will facilitate use of \cref{not_opposite_component}.
A \word{word} of length $\ell$ is $\mathbf w \in [r]^\ell$.
Let $\mathbf w + \mathbf w'$ denote the concatenation of two words $\mathbf w$ and $\mathbf w'$, and define the \word{support} of $\mathbf w$, denoted $\supp(\mathbf w)$, to be the subset of $[r]$ that appears in $\mathbf w$.
A \word{subword} of $\mathbf w$ is any word formed by deleting some entries of $\mathbf w$, and a \word{segment} of $\mathbf w$ is a subword of the form $\mathbf w' = (w_k, w_{k+1}, \ldots, w_{k'-1}, w_{k'})$ for some $1 \leq k \leq k'\leq \ell$.
Each word $\mathbf w = (w_1, \ldots, w_\ell)$ has an associated group element $s_{\mathbf w} = s_{w_1} s_{w_2} \cdots s_{w_\ell} \in W$, where $s_i$ is the simple reflection corresponding to $s_{\alpha_i}$.

\begin{definition}
  A word $\mathbf w$ is \word{efficient} on an element $u \in \mathcal S$ if there is no subword $\mathbf w'$ of $\mathbf w$ such that $s_{\mathbf w}(u) = s_{\mathbf w'}(u)$.
  Write $\SR(u)$ for the set of words efficient on $u \in \mathcal S$, and let $\SR = \cup_{u \in \mathcal S} \SR(u)$.
\end{definition}
We will use the following fact freely and without mention.
\begin{proposition}\label{segmenthereditary}
  If $\mathbf w \in \SR$, then all segments of $\mathbf w$ are in $\SR$.
\end{proposition}
Let $\Gamma$ be the directed graph with vertex set $[r]$, and an edge $i \to j$ whenever the Cartan integer $\ip{\alpha_i^\vee, \alpha_j}$ is odd. This graph is essentially the Dynkin diagram associated to $W$, and we will always label its vertices as in \cite[PLATES I-IX]{B02}.
\begin{proposition}\label{neighboraction}
  If $u = \sum_{j=1}^r u_j \lambda_j^\vee \in \mathcal S$, $\mathbf w \in \SR(u)$, and $i \in [r]$, then
  \[
    s_{\mathbf w}(u)_i = u_i + \sum_{\substack{j \to i\\ \text{is an edge of $\Gamma$}}} \#\{k \in [\ell] : w_k = j\} \lambda_j^\vee \in \Lambda^\vee / 2 \Lambda^\vee.
  \]
\end{proposition}
\begin{proof}
  We have
  \[
    s_i(\lambda_i^\vee)
    = \lambda_i^\vee - \ip{\alpha_i^\vee, \lambda_i^\vee} \alpha_i
    = \lambda_i^\vee - \ip{\alpha_i, \lambda_i^\vee}\alpha_i^\vee
    = \lambda_i^\vee - \sum_{j=1}^r \ip{\alpha_i^\vee, \alpha_j} \lambda_j^\vee.
  \]
  Consequently, for $u = \sum_{j=1}^r u_j \lambda_j^\vee \in \mathcal S = \Lambda^\vee / 2\Lambda^\vee$ with $u_i = 1$,
  \[
    s_i(u)_j = \begin{cases} u_j, & \ip{\alpha_i^\vee, \alpha_j} \equiv 0 \mod 2 \\
      u_j + 1, & \ip{\alpha_i^\vee, \alpha_j} \equiv 0 \mod 2.
    \end{cases}
  \]
  If $\mathbf w = (w_1, \ldots, w_\ell) \in \SR(u)$, then for each $1 \leq i \leq \ell$, $s_{w_i} \in \SR(s_{w_{i+1}}s_{w_{i+2}} \cdots s_{w_\ell}(u))$, so we must have $(s_{w_{i+1}}s_{w_{i+2}} \cdots s_{w_\ell}u)_{i} = 1$ (otherwise, we could delete $w_i$ to make a shorter word $\mathbf w'$ with $s_{\mathbf w'}(u) = s_{\mathbf w}(u)$).
  Consequently, we may sum the contributions of the individual simple reflection to obtain the desired formula.
\end{proof}
\begin{example}\label{forbidden}
  By \cref{segmenthereditary}, if a word $\mathbf w$ has a segment that is not efficient, then $\mathbf w$ is not efficient. Becuase of this, it will be useful to record some examples of words that are not efficient.
  \begin{itemize}
    \item The word $(i,i)$ is never efficient because $s_i s_i = 1$.
    \item If $i < j$ are neighbors in the Dynkin diagram of $W$, then $(i,j,i)$ is not efficient.
    If $W$ is of type $B_r$ and $j=r$, then $(i,j,i)$ is not efficient because $s_r$ acts trivially on $\mathcal S$; one argues similarly when $W$ is of type $F_4$ and $j=3$.
    In all other cases, the Cartan integer $\ip{\alpha_i^\vee, \alpha_j}$ is odd, so one of the two copies of $s_i$ in $s_i s_j s_i$ will act trivially unless $s_j$ acts trivially.
    \item If $W$ is of type $D_r$, then $(r-2, r-1, r)$ is not efficient. This is plainly true for any vector $u$ with $u_r$ or $u_{r-1}$ not equal to 1, and if $u_r = u_{r-1} = 1$, then the action of $s_{r-2} s_{r-1} s_r$ on $\mathcal S$ is the same as that of $s_{r-2}$.
  \end{itemize}
\end{example}
Say that $\mathbf w$ and $\mathbf w'$ differ by a \word{commutation} if $\mathbf w'$ can be obtained from $\mathbf w$ by replacing a segment $(i,j)$ of $\mathbf w$ such that $s_i s_j = s_j s_i$ with the segment $(j, i)$.
Write $\mathbf w \approx \mathbf w'$ if $\mathbf w'$ can be obtained from $\mathbf w$ by a sequence of commutations; in this case, $s_{\mathbf w} = s_{\mathbf w'}$.

If $\mathbf w$ and $\mathbf w'$ are words such that $\Gamma$ has no edges between $\supp(\mathbf w)$ and $\supp(\mathbf w')$, then $\mathbf w + \mathbf w' \approx \mathbf w' + \mathbf w$. Consequently, any word $\mathbf w$ can be written $\mathbf w \approx \mathbf w_1 + \mathbf w_2 + \cdots + \mathbf w_k$, where the supports of the $\mathbf w_i$'s are pairwise disjoint, and $\Gamma$ has no edges between $\supp(\mathbf w_i)$ and $\supp(\mathbf w_j)$ for $i \neq j$.
In this case, say $\mathbf w_1, \ldots, \mathbf w_k$ are \word{components} of $\mathbf w$.
We will use freely (and perhaps without mention) the following properties.
\begin{proposition}
  Let $u \in \mathcal S$.
  \begin{enumerate}
    \item If $\mathbf w \approx \mathbf w'$, then $\mathbf w \in \SR(u)$ if and only if $\mathbf w' \in \SR(u)$.
    \item $\mathbf w \in \SR(u)$ if and only if all components of $\mathbf w$ are in $\SR(u)$
  \end{enumerate}
\end{proposition}
The next two results will be crucial for the proof of \cref{gluedspacelemma}.
\begin{lemma}\label{onetime}
  Suppose $W$ is of classical type.
  If $\mathbf w \in \SR$, then $\min \left(\supp(\mathbf w)\right)$ appears exactly once in $\mathbf w$.
\end{lemma}
\begin{proof}
  We induct on the length of $\mathbf w$.
  The result plainly holds for words of length at most two.
  For length greater than two, suppose towards a contradiction that $i:=\min \left(\supp(\mathbf w)\right)$ appears at least twice in $\mathbf w$.
  Write
  \[
    \mathbf w \approx \mathbf w''' + (i) + \mathbf w'' + (i) + \mathbf w',
  \]
  with $i \not \in \supp(\mathbf w') \cup \supp(\mathbf w'')$.
  If $\supp(\mathbf w'')$ contains no neighbors of $i$, then up to commutation, $(i,i)$ is a segment of $\mathbf w$.
  If $\supp(\mathbf w'')$ contains one neighbor $j$ of $i$, then $j = \min\left( \supp(\mathbf w'')\right)$. By the induction hypothesis, $j$ appears exactly once in $\mathbf w''$, so up to commutation, $(i,j,i)$ is a segment of $\mathbf w$.
  Finally, if $\supp(\mathbf w'')$ contains two neighbors of $i$, then $W$ is of type $D_r$ and $i=r-2$.
  The two reflections $s_{r-1}$ and $s_r$ commute with all reflections other than $s_{r-2}$; therefore, up to commmutation, $(r-2, r-1, r)$ is a segment of $\mathbf w$.

  In each case, the segment obtained is not efficient per \cref{forbidden}, contradicting the fact that $\mathbf w$ is efficient.
\end{proof}
Given $u = \sum_{i=1}^r u_i \lambda_i^\vee \in \mathcal S$, let $\supp(u) := \{i : u_i \neq 0\}$ (we trust our repeated use of this notation in different contexts will cause no confusion).
\begin{lemma}\label{reductionlem}
  Suppose $W$ is of classical type.
  If $\mathbf w \in \sr(u)$ and $m := \min \left(\supp(\mathbf w)\right) \in \supp(u)$, then
    $\mathbf w \approx \mathbf w' + (m)$, with $\supp(\mathbf w') = \supp(\mathbf w) \setminus m$.
\end{lemma}
\begin{proof}
  By \cref{onetime},
    $\mathbf w \approx \mathbf w''' + (m) + \mathbf w''$,
  with $m \not \in \supp(\mathbf w''') \cup \supp(\mathbf w'')$.
  If there is a neighbor of $m$ in $\mathbf w''$, then $\min\left(\supp(\mathbf w'')\right)$ is a neighbor of $m$, and occurs exactly once by \cref{onetime}.
  However, by \cref{neighboraction} and the fact that $m \in \supp(u)$, the total number of neighbors of $m$ in $\mathbf w''$ must be even.

  These facts imply that $\supp(\mathbf w'')$ contains no neighbors of $m$, except perhaps when $W$ is of type $D_r$, $m= r-2$, and both $r, r-1 \in \supp(\mathbf w'')$. In this case, up to commutation, $(r-2, r-1, r)$ appears in $\mathbf w$, contrary to the fact that $\mathbf w$ is efficient. We conclude that $\mathbf w''$ contains no neighbors of $m$, so
  \[
    \mathbf w \approx \mathbf w''' + (m) + \mathbf w'' \approx \mathbf w''' + \mathbf w'' + (m) = \mathbf w' + (m).\qedhere
  \]
\end{proof}

\subsection{}
The combinatorics of \cref{words} is here applied to analyze irreducible components of the glued spaces $\mathbf{G}_u$, culminating in a proof of \cref{gluedspacelemma}.

Combining \cref{quotientmodel} and discussion of \cref{sec:faces}, we learn
that $H_IF_J$ is a face of $\mathbf{G}_u$ if and only if there is $w \in W_I$ and $v \in (\Z/2)\gen{\lambda_j^\vee : j \in J} \subset \mathcal S$  such that $w(u) + v > u$.
Moreover, the face $H_I F_J$ is a component of $\mathbf{G}_u$ if and only if the pair $(I, J)$ is minimal among the pairs $(I', J')$ such that $H_{I'}F_{J'}$ is a face of $\mathbf{G}_u$.
Minimality implies:
\begin{proposition}\label{witnessprop}
  If $H_IF_J$ is an irreducible component of $\mathbf{G}_u$, then there is a word $\mathbf w \in \SR(u)$ such that $\supp(\mathbf w) = I$ and $s_{\mathbf w}(u) + \sum_{j \in J} \lambda_j^\vee > u$.
  Moreover, either $(I, J) = (\emptyset, \{j\})$ for some $j \not \in \supp(u)$, or
  \[ J = \left\{i \in \supp(u) \setminus \supp(wu) :  i < \min(\supp(wu) \setminus \supp(u))\right\}. \]
\end{proposition}
\begin{proof}
  If $H_I F_J$ is a face of $\mathbf{G}_u$, then there is $w \in W_I $ and $v \in (\Z/2)\gen{\lambda_j^\vee : j \in J}$ such that $w(u) + v > u$.
  If $J$ contains some $j \not \in \supp(u)$, then we may take $w$ to be the identity and $v = \lambda_j^\vee$, so $(I, J) = (\emptyset, \{j\})$ by minimality.

  Otherwise, $J \subset \supp(u)$.
  Since $w \in W_I$, there is a word $\mathbf w$ such that $s_{\mathbf w}(u) = w(u)$ and $\supp(\mathbf w) \subset I$.
  By deleting elements from $\mathbf w$, we may assume that $\mathbf w \in \SR(u)$.

  If $\supp(\mathbf w) \subsetneq I$ or $v \neq \sum_{k \in J} \lambda_k^\vee$, then $(I, J)$ is not minimal, contrary to the fact that $H_I F_J$ is an irreducible component; therefore, $\supp(\mathbf w) = I$ and $v = \sum_{k \in J} \lambda_k^\vee$.
  Likewise, if there is $j \in J$ such that either $w(u)_j = 1$ or such that $j \geq \min \{k : w(u)_k = 1, u_k = 0\}$, then
  $w(u) + v - \lambda_j^\vee > u$,  so $(I, J \setminus j)$ also corresponds to a face of $\mathbf{G}_u$, contrary to the assumption that $H_I F_J$ is an irreducible component.
\end{proof}
In the circumstances of \cref{witnessprop}, call $\mathbf w$ a \word{witness} for $H_IF_J$.

\begin{lemma}\label{supportlem}
  If $W$ is of classical type and
  $H_I F_J$ is an irreducible component of $\mathbf{G}_u$, then $I \subset \supp(u)$.
\end{lemma}
\begin{proof}
  We prove for each of the four types, bootstrapping from type $A_r$.
  \begin{description}
    \item[Type $A_r$]
          Suppose towards a contradiction that there is $m \in I \setminus \supp(u)$.
          Let $\mathbf w$ be a witness for $H_IF_J$, and let $\mathbf w'$ be the connected component of $\mathbf w$ containing $m$.

          Since $\mathbf w'$ is a connected component, $\supp(\mathbf w') = \{k, k+1, \ldots, \ell\}$ for some $k \leq m \leq \ell$.
          If $j \in \supp(\mathbf w') \cap \supp(u)$ and $j-1 \not \in \supp(u)$, then $I = \{j\}$ and $J = \emptyset$ because $H_I F_J$ is a component, contrary to the fact that $m \in I$; therefore, if $j \in \supp(\mathbf w') \cap \supp(u)$, then $j-1 \in \supp(u)$.
          It follows that there is $k' < m$ such that
          \[
          \supp(\mathbf w') \cap \supp(u) = \{k, k+1, \ldots, k'\}.
          \]
          By \cref{reductionlem},
          \[
          \mathbf w' \approx \mathbf w'' + (k),
          \]
          where $\supp(\mathbf w'') = \supp(\mathbf w') \setminus k$.
          Since $k' < m$ and $m \in \supp(\mathbf w')$, $k+1 \in \supp(\mathbf w')$, hence $k+1$ appears exactly once in $\supp(\mathbf w'')$ by \cref{onetime}.
          By \cref{neighboraction}, $s_{\mathbf w}(u)_k = 0$, but $u_k = 1$.

          If $k > c := \min(\supp(wu) \setminus \supp(u))$, then the word $\mathbf w'''$ formed by deleteing $\mathbf w'$ from $\mathbf w$ satisfies $s_{\mathbf w'''}(u) + v > u$ and $\supp(\mathbf w''') \subsetneq I$, contrary to the fact that $H_I F_J$ is an irreducible component; therefore, $k \leq c$.
          Since $k \neq c$, it must be the case that $k < c$; therefore, $k \in J$ by \cref{witnessprop}.
          This is a contradiction because $I \cap J = \emptyset$.
    \item[Type $B_r$] Suppose towards a contradiction that $m \in I \setminus\supp(u)$.
          Since $s_r$ acts trivially on $\mathcal S$, $r$ is not in the support of any efficient word; therefore, the conclusion holds automatically if $m=r$.
          Otherwise, if $m < r$, then the Type A proof goes through verbatim.
    \item[Type $C_r$]
          Suppose towards a contradiction that $m \in I\setminus\supp(u)$.
          No element of $W$ alters the $r$th coordinate of $u$, so $r$ appears in no word efficient on $u$ when $u_r = 0$.
          This handles the case $m = r$.
          Otherwise, if $m < r$, then the Type A proof goes through verbatim.
    \item[Type $D_r$]
          Suppose towards a contradiction that $m \in I \setminus \supp(u)$.
          If $m \leq r-2$, then the Type A proof goes through verbatim.

          Otherwise, suppose $m \in \{r-1, r\}$.
          Let $\mathbf w$ be a witness for $H_IF_J$ and let $\mathbf w'$ be the connected component of $m$.
          Since $\mathbf w'$ is efficient on $u$, $r-2$ appears to the right of $m$ in $\mathbf w'$.
          Arguing as in Type~A,
          \[ (\supp(\mathbf w') \cap \supp(u)) \setminus \{r, r-1\} = \{k, k+1, \ldots, r-2\} \]
          for some $k \leq r-2$.
          Suppose towards a contradiction that $k < r-2$. By \cref{reductionlem},
          \[
          \mathbf w' \approx \mathbf w'' + (k),
          \]
          where $\mathbf w''$ has no copies of $k$ and exactly one copy of $k+1$.
          By \cref{neighboraction}, it follows that $s_{\mathbf w}(u)_k = 0$ but $u_k = 1$, so $k \in J$, contrary to the fact that $I \cap J = \emptyset$.
          Consequently,
          \[ \mathbf w' \approx \mathbf w'' + (r-2), \]
          where $\mathbf w'' \in \{ (r), (r-1), (r, r-1), (r-1, r)\}$.
          If $\mathbf w'' \in \{(r), (r-1)\}$, then $s_{\mathbf w}(u)_{r-2} = 0$ but $u_{r-2} = 1$, so $r-2 \in J$, contrary to the fact that $I \cap J = \emptyset$. Therefore, $\mathbf w'' \approx (r-1, r)$ and $u_{r-1} = u_r = 0$.
          But then $s_{r-1}s_r s_{r-2}(u) = s_{r-2}(u)$, so $\mathbf w$ is not efficient on $u$, a contradiction.
  \end{description}
\end{proof}

\begin{proof}[Proof of \cref{gluedspacelemma}]
  Let $u \in \mathcal S$ with $u \neq \sum_{i=1}^r \lambda_i^\vee $, and let $j \in [r] \setminus \supp(u)$. The face $F_j$ is an irreducible component of $\mathbf G_u$.

  If $W$ is of classical type, and $H_IF_J$ is a component of $\mathbf G_u$, then $j \not \in I$ by \cref{supportlem}; therefore, $H_I F_J \cap F_j \neq \emptyset$.
  By \cref{not_opposite_component}, it follows that $\mathbf G_u$ is contractible.

  For $W$ of exceptional type ($E_6$, $E_7$, $E_8$, $F_4$, $G_2$), we checked that $\mathbf G_u$ is contractible by using a computer to verify the hypothesis of \cref{not_opposite_component}. Our code is available at \cite{Colin_Crowley_and_Tao_Gong_and_Connor_Simpson_A_tool_for}.
\end{proof}

\printbibliography

\end{document}